\documentclass[leqno,12pt]{article}
\usepackage{amssymb,amsmath,amsfonts,amsthm}
\usepackage{authblk}
\usepackage{epsfig,color,mathrsfs}
\usepackage{graphicx}
\usepackage{amscd}
\usepackage{caption}
\usepackage{mathrsfs,bbm}
\usepackage{multicol}
\usepackage{color}

\numberwithin{equation}{section}

\newcommand{\R}{\mathbb R}
\newcommand{\Ro}{\mathbb R^N\setminus \overline B_1}
\newcommand{\Lg}{{\mathcal L}_g}
\newtheorem{theorem}{Theorem}[section]
\newtheorem{corollary}[theorem]{Corollary}
\newtheorem{lemma}[theorem]{Lemma}
\newtheorem{proposition}[theorem]{Proposition}

 \usepackage{amsmath,amsthm,amsfonts,amssymb}
\usepackage[mathscr]{euscript}
 \usepackage{hyperref} 
 \usepackage{mathrsfs} 
\usepackage{graphicx}
\usepackage{color}
  \usepackage{epsfig}

\begin{document}
\title{\vskip-0.3in Double phase inequalities with convolution nonlinearity in exterior domains} 

\author[ ]{Marius Ghergu}

\affil[ ]{School of Mathematics and Statistics, University College Dublin}
\affil[ ]{Belfield Campus, Dublin 4, Ireland}
\affil[ ]{and}
\affil[ ]{Institute of Mathematics Simion Stoilow of the Romanian Academy} 
\affil[ ]{21 Calea Grivitei St., 010702 Bucharest, Romania}
\affil[ ]{E-mail: {\tt marius.ghergu@ucd.ie}}

\maketitle

\begin{abstract} 
We discuss the existence of $C^1$-solutions for two related double phase inequalities:
\begin{equation*}
{\mathcal L}_g u\pm \Delta_s u\geq (|x|^{-\alpha}*u^p)u^q \quad\mbox{ in }\mathbb R^N\setminus \overline B_1, N\geq 1,\tag{$P^\pm$}
\end{equation*}
in which $\Delta_s u:={\rm div}\big(|\nabla u|^{s-2}\nabla u\big)$ is the $s$-Laplace operator, $s>1$, and
$$
{\mathcal L}_g u:= -{\rm div}\Big(|\nabla u|^{m-2}g(|\nabla u|)\nabla u\Big),\quad m>s>1,
$$
where $g:[0, \infty)\to (0, \infty)$ is a $C^1(0, \infty)\cap C[0, \infty)$  non-increasing function with some specific behaviour near the origin. In the above context, the general form of ${\mathcal L}_g u$ includes the case of $m$-Laplace and $m$-mean curvature operator.
Our study reveals a sharp distinction between $(P^+)$ and $(P^-)$. Precisely, we show that the inequality $(P^+)$ has solutions for all $m>s>1$ and $q>s-1$. In contrast, $(P^-)$ has solutions if and only if $p$ and $q$ are sufficiently large. We also link the solvability of $(P^-)$ with that of the corresponding equation ${\mathcal L}_g u- \Delta_s u= (|x|^{-\alpha}*u^p)u^q$  in $\mathbb R^N\setminus \overline B_1$, for which we derive optimal conditions in terms of $p, q, \alpha, s$ and $N$.
The approach combines integral estimates with a new sub and supersolution method that accounts for the presence of the convolution term.  
\end{abstract}


\medskip

\section{Introduction}

This article is concerned with two double phase elliptic inequalities of the type
\begin{equation*}
\Lg u+\Delta_s u\geq (|x|^{-\alpha}*u^p)u^q \quad\mbox{ in }\Ro,N\geq 1, \tag{$P^+$}
\end{equation*}
and
\begin{equation*}
\Lg u-\Delta_s u\geq (|x|^{-\alpha}*u^p)u^q \quad\mbox{ in }\Ro, N\geq 1,\tag{$P^-$}
\end{equation*}
in which $\Delta_s u:={\rm div}\big(|\nabla u|^{s-2}\nabla u\big)$ is the $s$-Laplace operator, $s>1$, and
$$
\Lg u:= -{\rm div}\Big(|\nabla u|^{m-2}g(|\nabla u|)\nabla u\Big),\quad m>s>1,
$$
where $g:[0, \infty)\to (0, \infty)$ is a $C^1(0, \infty)\cap C[0, \infty)$  non-increasing function such that 
\begin{equation}\label{cong}
\lim_{t\to 0^+}\frac{tg'(t)}{g(t)}=0.
\end{equation}
Typical examples of functions $g$ with the above property are:
\begin{itemize} 
\item $g(t)\equiv 1$ for which $\Lg u$ becomes the $m$-Laplace operator and thus $(P^+)$ and $(P^-)$ read 
\begin{equation}\label{lap1}
-\Delta_m u\pm \Delta_s u\geq (|x|^{-\alpha}*u^p)u^q \quad\mbox{ in }\Ro\, , m>s>1.
\end{equation}

\item $g(t)=\frac{1}{\sqrt{1+t^m}}$, $m\geq 2$, for which $\Lg u$ becomes the $m$-mean curvature operator and thus $(P^+)$ and $(P^-)$ read 
\begin{equation}\label{lap2}
-H_m u\pm \Delta_s u\geq (|x|^{-\alpha}*u^p)u^q \quad\mbox{ in }\Ro, m>s>1,
\end{equation}
where 
\begin{equation}\label{mcv}
H_m u:={\rm div}\Bigg( \frac{|\nabla u|^{m-2}}{\sqrt{1+|\nabla u|^m}}\nabla u\Bigg).
\end{equation}
\item $g(t)=\log^{-\theta}(2+t^k) $, $\theta, k>0$.
\end{itemize} 
In $(P^\pm)$ we assume $\alpha\in (0, N)$, $p, q>0$, while $|x|^{-\alpha}*u^p$ denotes the convolution operation over $\R^N\setminus B_1$, that is,
$$ 
|x|^{-\alpha}*u^p:=\int\limits_{\R^N\setminus B_1} \frac{u^p(y)}{|x-y|^{-\alpha}} dy \quad\mbox{ for all }x\in \Ro. 
$$ 

We say that a function $u\in C(\Ro)\cap W_{loc}^{1, m}(\Ro)$ is a positive  solution of $(P^-)$ if:
\begin{itemize}
\item $u>0$ and
\begin{equation}\label{c1}
\int\limits_{\R^N\setminus B_1}\frac{u^{p}(y)}{1+|y|^{\alpha}}dy< \infty.
\end{equation}
\smallskip

\item for any $\phi \in C_{c}^{\infty}(\Ro)$, $\phi\geq 0$, we have
\end{itemize}
\begin{equation}\label{c2}
\int\limits_{\R^N\setminus B_1} \Big(|\nabla u|^{m-2}g(|\nabla u|)+|\nabla u|^{s-2}\Big) \nabla u\cdot \nabla \phi \geq \int\limits_{\R^N\setminus B_1} (|x|^{-\alpha}*u^p)u^q \phi.
\end{equation}
Switching the addition symbol in \eqref{c2} to subtraction  we obtain the definition of a positive solution to $(P^+)$.

Currently, there is an increasing interest in double phase problems \cite{CJJ26, CPR26, TZ26, WS25, Z26, ZLP25} due to their relevance in physics. In such a setting, double phase models are linked to heterogeneous structures and their anisotropic properties. On the other hand, PDEs featuring convolution terms were first introduced by Hartree \cite{Har28a, Har28b, Har28c} as mathematical models in quantum physics. In \cite{GKS20} the following inequality is considered
$$
\Lg u\geq (|x|^{-\alpha}*u^p)u^q\quad\mbox{ in }\R^N\setminus \overline B_1.
$$
Similar problems were discussed in punctured balls \cite{FG20, GZ26}, cone like domains \cite{GZ23},  and even in a time dependent settings \cite{FG22a, FG22b}.

The present article brings together double phase operators and nonlinear convolution terms in the form of $(P^+)$ and $(P^-)$. Problems that feature nonlocal operators or convolution terms can also be found in \cite{LR26, MNV24, NS26}.

Our study reveals a striking difference between $(P^+)$ and $(P^-)$. As concerns $(P^+)$ we have the following result.

\begin{theorem}\label{tp}
Assume $g$ satisfies \eqref{cong}. Then, for all $m>s>1$ and $q>s-1$, the inequality $(P^+)$ has a positive solution. 
\end{theorem}

We now turn to the study of inequality $(P^-)$. In this framework, we assume that apart from \eqref{cong}, the function $g$ satisfies 
\begin{equation}\label{cong1}
\frac{tg'(t)}{g(t)}+m-1\geq 0 \quad\mbox{ for all }t>0.
\end{equation}
One can easily check that the functions $g(t)=1$ and $g(t)=\frac{1}{\sqrt{1+t^m}}$, $m\geq 2$ which yield the $m$-Laplace and $m$-mean curvature operators respectively, satisfy both \eqref{cong} and \eqref{cong1}.

We first show that no positive solutions exist if $m>s\geq N$. 
\begin{theorem}\label{tm1}
Assume $g$ satisfies \eqref{cong}, \eqref{cong1} and $m>s\geq N$. Then, the inequality $(P^-)$ has no positive solutions. 
\end{theorem}

Our next result discusses the existence of solutions to $(P^-)$. 

\begin{theorem}\label{tm2}
Assume $g$ satisfies \eqref{cong}, \eqref{cong1}, $m>s>1$, $s<N$, $q\geq \alpha-1$ and $\min\{p,q\}>m-1$. 
The following statements are equivalent:
\begin{enumerate}
\item[(i)] The inequality $(P^-)$ has a positive solution;
\item[(ii)] The equation 
\begin{equation}\label{eqpq}
\Lg-\Delta_s u=(|x|^{-\alpha}*u^p)u^q
\quad\mbox{ in }\R^N\setminus \overline B_1,
\end{equation}  
has a positive $C^1$-solution;
\item[(iii)] The exponents $p$ and $q$ satisfy
\begin{equation}\label{pq}
\min\{p, q\}>\frac{(s-1)(N-\alpha)}{N-s} \quad\mbox{ and }\quad p+q>\frac{(s-1)(2N-\alpha)}{N-s}.
\end{equation}
\end{enumerate}
\end{theorem}

Theorem \ref{tm2} provides a connection between the inequality $(P^-)$ and the corresponding equation \eqref{eqpq} and states that the existence of a solution to both $(P^-)$ and \eqref{eqpq} requires the same conditions \eqref{pq} on the data. In order to study the equation \eqref{eqpq} we devise a new sub and supersolution methodsuitable for the Orlicz spaces which is different from that provided in \cite{F18}. From Theorem \ref{tm2} the following result follows immediately.

\begin{corollary}\label{cor1}
Under the same conditions as in Theorem \ref{tm2}, the following statements are equivalent:
\begin{enumerate}
\item[(i)] The equation 
$$
-\Delta_m u-\Delta_s u=(|x|^{-\alpha}*u^p)u^q \quad\mbox{ in }\R^N\setminus \overline B_1,
$$
has a positive $C^1$-solution;
\item[(ii)] The equation 
$$
-H_m u-\Delta_s u=(|x|^{-\alpha}*u^p)u^q \quad\mbox{ in }\R^N\setminus \overline B_1,
$$
where $m\geq 2$ and $H_m$ is the operator in \eqref{mcv}, 
has a positive $C^1$-solution;
\item[(iii)] The exponents $p$ and $q$ satisfy \eqref{pq}. 
\end{enumerate}

\end{corollary}

The remaining of this article is organised as follows. Section 2 contains some preliminary results related to integral estimates and properties of the operator $\Lg-\Delta_s$. The proofs of the above main results are given in sections 3, 4 and 5 respectively. Throughout this work, by $C, C_1, C_2, c_1, c_2, \dots, $ we denote generic constants whose value may change on each occurrence. 

\section{Preliminary results}

In this section we gather some preliminary results related to the operator $\Lg u-\Delta_s u$. We start with the following elementary result. 
\begin{lemma}\label{lh}
Let $\kappa\geq 1>a>0$ and $h:[0, \infty)\to \R$, $h(t)=t^\kappa-Dt^a-E$. Suppose there exists $c>0$ such that 
\begin{equation}\label{condh}
E^{1-\frac{a}{\kappa}}\geq cD.
\end{equation}
Then, the equation $h(t)=0$ has a unique solution $t_0>0$ and $h(t)<0$ for all $0<t<t_0$, while $h(t)>0$ for all $t>t_0$. Furthermore, there exists $\lambda>1$ depending on $c$ only such that $E<t_0<\lambda E^{\frac{1}{\kappa}}$.
\end{lemma}
\begin{proof}
We have $h'(t)=t^{a-1}(\kappa t^{\kappa-a}-aD)$ and there exists a unique solution of the equation $h'(t)=0$. Since $h(0)=-E<0$, it also follows that the equation $h(t)=0$ has a unique solution $t_0>0$. We note that $h(E)<0$ while for large $\lambda>1$ we have
$$
\frac{\lambda^\kappa-1}{\lambda^a}>\frac{1}{c},$$
and thus, by \eqref{condh} one has
$$
h(\lambda E^{\frac{1}{\kappa}})=\lambda^a E^{\frac{a}{\kappa}} \Big(\frac{\lambda^\kappa-1}{\lambda^a}E^{1-\frac{a}{\kappa}}-D  \Big)>0.
$$
This shows that the unique solution $t_0>0$ of the equation $h(t)=0$ satisfies $E<t_0<\lambda E^{\frac{1}{\kappa}}$, which concludes the proof.
\end{proof}

We shall use the following result from \cite[Lemma 3.7]{GKS20} (see also \cite{Gbook22, GT16book} for further results).
\begin{lemma}\label{lbas}
Let $\alpha\in (0,N)$ and $\rho\geq  1/2$.

\begin{enumerate}
\item[\rm (i)]
If $f(x)\geq c|x|^{-\beta}$ in $\R^N\setminus B_{1}$ for some $c,\beta>0$, then
$$ \left\{
\begin{aligned}
&  |x|^{-\alpha}*f=\infty && \quad\mbox{ if } \; \beta \leq N-\alpha \\
&  |x|^{-\alpha}*f\geq C|x|^{N-\alpha-\beta} && \quad\mbox{ if } \; \beta> N-\alpha
\end{aligned}
\right. \quad\mbox{ in }\; \R^N\setminus B_{1},
$$
for some constant $C>0$.

\item[\rm (ii)]
If $f\in C(\R^N\setminus B_\rho)$ and $\displaystyle f(x)\leq  c|x|^{-\beta}$ in $\R^N\setminus B_{\rho}$ for some $c>0$ and $\beta>N-\alpha>0$, then for some constant $C>0$ we have
$$ 
\int\limits_{\R^N\setminus B_{\rho}}\frac{f(y)}{|x-y|^{\alpha}} dy \leq 
\left\{
\begin{aligned}
&C |x|^{N-\alpha-\beta} &&\quad\mbox{ if }\; N-\alpha<\beta<N\\
&C|x|^{-\alpha}\log|x|   &&\quad\mbox{ if }\; \beta=N\\
&C|x|^{-\alpha} &&\quad\mbox{ if }\; \beta>N
\end{aligned}
\right.
\quad\mbox{in }\R^N\setminus \overline B_{2\rho}.
$$
\end{enumerate}
\end{lemma}

The operator $\Lg-\Delta_s$ enjoys the following comparison principle in bounded domains.

\begin{proposition}\label{cp}
Assume $g$ satisfies \eqref{cong1} and let $\Omega\subset \R^N$ be a bounded domain. Suppose $u,v\in W^{1,1}(\Omega)\cap C(\Omega)$ satisfy 
$$
\begin{aligned}
\Lg u-\Delta_s u &\geq \Lg v -\Delta_s v \quad\mbox{ in } \Omega, \\
u&\geq v \quad\mbox{ on } \partial\Omega.
\end{aligned}
$$ 
Then $u\geq v$ in $\Omega$.
\end{proposition}
\begin{proof} Let us first note that
$$
\Lg w-\Delta_s w=-{\rm div}\big[A(|\nabla w|)\nabla w  \big],
$$
where $A(t)=t^{m-2}g(t)+t^{s-2}$. Furthermore, from \eqref{cong1} we have $t\longmapsto tA(t)$ is increasing.

Suppose $u\geq v$ does not hold in $\Omega$. Then, for all $\varepsilon>0$ small,  the set $\Omega_\varepsilon=\{x\in \Omega:u(x)<v(x)-\varepsilon\}$ is nonempty. Obviously, $\overline \Omega_\varepsilon\subset \Omega$ and by the continuity of $u$ and $v$ it follows that $u=v-\varepsilon$ on $\partial \Omega_\varepsilon$.

Let $\psi:\R\to [0,\infty)$ be a nondecreasing function with the properties 
$$\psi \equiv 0\quad\mbox{ on }(-\infty, 0]\quad\mbox{ and }\quad \psi'>0 \quad\mbox{ on }(0,\infty).
$$ 
Then
$$
-\int_{\Omega_\varepsilon} {\rm div}\Big(A(|\nabla v|)\nabla v -A(|\nabla u|)\nabla u \Big)\psi(v-\varepsilon-u) dx\leq  0.
$$
By the divergence theorem \index{Theorem!divergence} this implies 
\begin{equation}\label{AA1}
\int_{\Omega_\varepsilon} \Big(A(|\nabla v|)\nabla v -A(|\nabla u|)\nabla u  \Big) \cdot(\nabla v-\nabla u)\psi'(v-\varepsilon-u)dx\leq 0.
\end{equation}
On the other hand, we note that
\begin{equation}\label{AA2}
\begin{aligned}
\Big(A(|\nabla v|)\nabla v -A(|\nabla u|)\nabla u  \Big) \cdot  (\nabla v&-\nabla u)\\
=&
\Big(A(|\nabla v|)|\nabla v| -A(|\nabla u|) |\nabla u|  \Big) (|\nabla v|-|\nabla u|)\\
&+\Big(A(|\nabla v|) +A(|\nabla u|)  \Big) (|\nabla v||\nabla u|-\nabla v\cdot \nabla u)\\
\geq & 0.
\end{aligned}
\end{equation}
Thus, \eqref{AA1} and \eqref{AA2} yield $\nabla u=\nabla v$ in $\Omega_\varepsilon$. This means $u-v=c$ in $\Omega_\varepsilon$, which is impossible since $u-v=\varepsilon$ on $\partial \Omega_\varepsilon$ and $u<v-\varepsilon$ in $\Omega_\varepsilon$.
Hence, for all $\varepsilon>0$ small we have $u\geq v-\varepsilon$ in $\Omega$. The conclusion follows now by letting $\varepsilon\to 0$.
\end{proof}

\section{Proof of Theorem \ref{tp}}

Let us first note that if $u=u(r)$ is a radially symmetric function, then
\begin{equation}\label{sv}
-\Delta_s u=\frac{|u'(r)|^{s-1}}{r}\Big\{\frac{ru''(r)}{u'(r)}(s-1)+N-1\Big\},
\end{equation}
and 
\begin{equation}\label{gv}
\Lg u=\frac{t^{m-1}g(t)}{r}\Big\{\frac{ru''(r)}{u'(r)}\Big(m-1+\frac{tg'(t)}{g(t)}\Big)+N-1\Big\},
\end{equation}
where $t=|u'(r)|$. In particular, if $u(x)=\kappa |x|^{-\gamma}$, then 
\begin{equation}\label{sl}
-\Delta_s u(x)=(\kappa \gamma)^{s-1}|x|^{-(\gamma+1)(s-1)-1}\big\{-(\gamma+1)(s-1)+N-1\big\}
\end{equation}
and
\begin{equation}\label{gl}
\Lg u(x)=-(\kappa \gamma)^{m-1}|x|^{-(\gamma+1)(m-1)-1}g(t)\Big\{(\gamma+1)\Big(m-1+\frac{tg'(t)}{g(t)}\Big)+N-1\Big\},
\end{equation}
where $t=\kappa \gamma|x|^{-\gamma-1}$. We take
$$
\gamma>\max\left\{\frac{N}{p}, \frac{N-s}{s-1}, \frac{s-\alpha}{q-s+1}  \right\}.
$$

Thus, for some positive constant $C_1>0$ we estimate
$$
\Delta_s u\geq C_1\kappa^{s-1} |x|^{-(\gamma+1)(s-1)-1}\quad\mbox{ in }\Ro.
$$
We have $t=\kappa \gamma|x|^{-\gamma-1}\to 0$ as $|x|\to \infty$. Thus, for $\kappa>0$ small and $|x|>1$, we have that $t$ is close to zero and thus $g(t)$ is close to $g(0)>0$. Using \eqref{cong} and \eqref{gl} we find
$$
\Lg u\geq -C_2 \kappa^{m-1} |x|^{-(\gamma+1)(m-1)-1} \quad\mbox{ in }\Ro,
$$
for some $C_2>0$, and thus
\begin{equation}\label{d1}
\begin{aligned}
\Lg u+\Delta_s u &\geq C_1\kappa^{s-1} |x|^{-(\gamma+1)(s-1)-1}-C_2 \kappa^{m-1} |x|^{-(\gamma+1)(m-1)-1}\\
&\geq C \kappa^{s-1} |x|^{-(\gamma+1)(s-1)-1} \quad\mbox{ in }\Ro,
\end{aligned}
\end{equation}
provided $\kappa>0$ is sufficiently small. On the other hand, by 
Lemma \ref{lbas}(ii) with $f=u^p$, $\rho=1/2$, and $\beta=p\gamma>N$, we see that
$$
\int_{\R^N\setminus B_1}\frac{u^p(y)}{|x-y|^\alpha} dy\leq C_0\kappa^p |x|^{-\alpha}\quad\mbox{ for all }x\in \Ro,
$$
where $C_0>0$ is independent of $\kappa$.
Hence,
\begin{equation}\label{d2}
(|x|^{-\alpha}*u^p)u^q\leq C_0\kappa^{p+q} |x|^{-\alpha-q\gamma} \quad\mbox{ for all }x\in \Ro.
\end{equation}
Since $\gamma>\frac{s-\alpha}{q-(s-1)}$, from \eqref{d1} and \eqref{d2} we may select $\kappa>0$ sufficiently small and deduce that $u=\kappa|x|^{-\gamma}$ is a solution of $(P^+)$ in $\Ro$. 
\qed

\section{Proof of Theorem \ref{tm1}}

Suppose by contradiction that $(P^-)$ has a positive solution $u$ in $\Ro$.

From \eqref{cong} there exists $\tilde t>0$ such that 
\begin{equation}\label{gg}
\frac{tg'(t)}{g(t)}>s-m\quad\mbox{ for all }0<t<\widetilde t.
\end{equation}

Let 
\begin{equation}\label{cc}
0<c<\min\Big\{\tilde t, \min_{\partial B_{2}}u\Big\}.
\end{equation}
For any $0<\varepsilon, \gamma<1$ define $v_{\varepsilon, \gamma}(x)=c|x|^{-\gamma}-\varepsilon$. Then, for $c>0$ small, from \eqref{sl} we have $-\Delta_s v_{\varepsilon, \gamma}\leq 0$ in $\R^N\setminus B_{2}$. Also, from \eqref{gl}  we compute
$$
\Lg v_{\varepsilon, \gamma}=-\frac{t^{m-1}g(t)}{r}\Big\{(\gamma+1)\Big(m-1+\frac{tg'(t)}{g(t)}\Big)-(N-1)\Big\},
$$
where $t=|v'_{\varepsilon, \gamma}(r)|=c\gamma r^{-\gamma}$ and $r=|x|$. Further, from \eqref{cc} we obtain
$$
t\leq c<\tilde t \quad\mbox{ for all }r=|x|\geq 2.
$$
Using \eqref{gg} it follows that 
$$
(\gamma+1)\Big(m-1+\frac{tg'(t)}{g(t)}\Big)-(N-1)> (\gamma+1)(s-1)-(N-1)>0,
$$
which yields 
\begin{equation}\label{cc1}
\Lg v_{\varepsilon, \gamma}-\Delta_s v_{\varepsilon, \gamma}\leq 0\quad\mbox{ in }\R^N\setminus B_{2}.
\end{equation} 
Since $v_{\varepsilon, \gamma}(x)\to -\varepsilon$ as $|x|\to \infty$, there exists $R>2$ such that $u>v_{\varepsilon, \gamma}$ in $\R^N\setminus B_R$. In particular this yields $u\geq v_{\varepsilon, \gamma}$ on $\partial B_R$ and from \eqref{cc} one has 
$$
v_{\varepsilon, \gamma}(x)=c|x|^{-\gamma}-\varepsilon\leq c<u \quad\mbox{ on }\partial B_{2}.
$$
This shows that $u, v_{\varepsilon, \gamma}$ satisfy the conditions in Lemma \ref{cp} in $\Omega=B_R\setminus B_{2}$ and thus $u\geq v_{\varepsilon, \gamma}$ in $B_R\setminus B_{2}$. This ultimately yields $u\geq c|x|^{-\gamma}-\varepsilon$ in $\R^N\setminus B_{2}$. We now let $\varepsilon\to 0$ and then $\gamma\to 0$ to derive $u\geq c$ in $\R^N\setminus B_{2}$. This however contradicts  condition \eqref{c1} and concludes our proof (see also Lemma \ref{lbas} (i)).
\qed

\section{Proof of Theorem \ref{tm2}}

A crucial point in our approach is the following a priori result.
\begin{lemma}\label{lap}
Let $u$ be a positive solution of $(P^-)$. Then, there exists a constant $C=C(N,u,m,s,p,q)>0$ such that 
\begin{equation}\label{d20}
\int_{B_{3R}\setminus B_{2R}}u^{\frac{p+q-s+1}{2}}(x)dx \leq CR^{\frac{N+\alpha-s}{2}} \quad\mbox{ for all }R>2.
\end{equation}
\end{lemma}
\begin{proof}
Let $\phi\in C^\infty_c(\Ro)$ be a standard cut-off function such that:
\begin{itemize}
\item $0\leq \phi\leq 1$ and supp$\,\phi\subset B_{4R}\setminus B_{R}$;
\item  $\phi=1$ in $B_{3R}\setminus B_{2R}$;
\item  $|\nabla \phi|\leq \frac{C}{R}$ in $\Ro$, where $C=C(N)>0$ is a constant.
\end{itemize}
We multiply the inequality $(P^-)$ by $u^{1-s}\phi^\lambda$ and obtain
$$
\int\limits_{\R^N\setminus B_1}(|x|^{-\alpha}*u^p)u^{q-s+1}\phi^\lambda\leq \int\limits_{\R^N\setminus B_1}\Big(|\nabla u|^{m-2}g(|\nabla u|)+|\nabla u|^{s-2}\Big) \nabla u\cdot \nabla \big(u^{1-s}\phi^\lambda\big),
$$
which yields
\begin{equation}\label{d3}
\begin{aligned}
\int\limits_{\R^N\setminus B_1}(|x|^{-\alpha}*u^p)& u^{q-s+1}\phi^\lambda+(s-1) \int\limits_{\R^N\setminus B_1}\Big(|\nabla u|^{m}g(|\nabla u|)+|\nabla u|^{s}\Big) u^{-s}\phi^\lambda\\
&\leq \lambda \int\limits_{\R^N\setminus B_1}  \Big(|\nabla u|^{m-1}g(|\nabla u|)+|\nabla u|^{s-1}\Big) 
u^{1-s}\phi^{\lambda-1}|\nabla \phi|.
\end{aligned}
\end{equation}
By Young's inequality we have 
$$
\begin{aligned}
|\nabla u|^{m-1}u^{1-s}\phi^{\lambda-1}|\nabla \phi|&= \Big( |\nabla u|^{m-1} u^{-\frac{s(m-1)}{m}}\phi^{\frac{\lambda(m-1)}{m}}\Big)\cdot \Big(|\nabla \phi|u^{\frac{m-s}{m}}\phi^{\frac{\lambda-m}{m}}\Big)\\
&\leq \frac{s-1}{2\lambda}|\nabla u|^{m} u^{-s}\phi^\lambda+C_1(s,m,\lambda)|\nabla \phi|^{m} u^{m-s}\phi^{\lambda-m}.
\end{aligned}
$$
Similarly,
$$
\begin{aligned}
|\nabla u|^{s-1}u^{1-s}\phi^{\lambda-1}|\nabla \phi|&\leq \frac{s-1}{2\lambda}|\nabla u|^{s} u^{-s}\phi^\lambda+C_2(s,\lambda)|\nabla \phi|^{s}\phi^{\lambda-s}.
\end{aligned}
$$
Using these last two estimates into \eqref{d3} and the fact that $g(|\nabla u|)\leq g(0)$ we find
\begin{equation}\label{d4}
\begin{aligned}
\int\limits_{\R^N\setminus B_1}(|x|^{-\alpha}*u^p) u^{q-s+1}\phi^\lambda+\frac{s-1}{2} & \int\limits_{\R^N\setminus B_1}\Big(|\nabla u|^{m}g(|\nabla u|)+|\nabla u|^{s}\Big) u^{-s}\phi^\lambda\\
&\leq C\int\limits_{\R^N\setminus B_1} |\nabla \phi|^{m} u^{m-s}\phi^{\lambda-m} +C\int\limits_{\R^N\setminus B_1}  |\nabla \phi|^{s}\phi^{\lambda-s},
\end{aligned}
\end{equation}
where $C=C(s, m, \lambda)>0$. Since $|\nabla \phi|\leq \frac{C}{R}$ in $\Ro$ and supp$\,\phi\subset B_{4R}\setminus B_{R}$, \eqref{d4} yields
\begin{equation}\label{d5}
\int\limits_{B_{4R}\setminus B_R}(|x|^{-\alpha}*u^p) u^{q-s+1}\phi^\lambda \leq CR^{-m}\int\limits_{B_{4R}\setminus B_R} u^{m-s}\phi^{\lambda-m} +CR^{N-s},
\end{equation}
for all $R>2$. 

For all $x,y\in B_{4R}\setminus B_R$ we have $|x-y|\leq |x|+|y|\leq 8R$ and then
$$
\begin{aligned}
(|x|^{-\alpha}*u^p)(x)&\geq \int\limits_{B_{4R}\setminus B_{R}}\frac{u^{p}(y)}{|x-y|^{\alpha}}dy\\
&\geq \int\limits_{B_{4R}\setminus B_{R}}\frac{u^{p}(y)}{(8R)^{\alpha}}dy\\
&\geq CR^{-\alpha}\int\limits_{B_{4R}\setminus B_R}u^{p}(y)dy.
\end{aligned}
$$
Using the above estimate in \eqref{d5} we derive
\begin{equation}\label{d6}
\Big(\int\limits_{B_{4R}\setminus B_R}u^{p}\Big) \Big(\int\limits_{B_{4R}\setminus B_R} u^{q-s+1}\phi^\lambda \Big) \leq CR^{\alpha-m}\int\limits_{B_{4R}\setminus B_R} u^{m-s}\phi^{\lambda-m} +CR^{N+\alpha-s},
\end{equation}
for all $R>2$. A H\"older's inequality applied to the left-hand side of \eqref{d6} together with the fact that $0\leq \phi\leq 1$ yields
\begin{equation}\label{d7}
\begin{aligned}
\Big(\int\limits_{B_{4R}\setminus B_R} u^{\frac{p+q-s+1}{2}}\phi^\lambda \Big)^2&\leq \Big(\int\limits_{B_{4R}\setminus B_R} u^{\frac{p+q-s+1}{2}}\phi^{\frac{\lambda}{2}} \Big)^2\\
&\leq \Big( \int\limits_{B_{4R}\setminus B_R}u^{p}\Big) \Big(\int\limits_{B_{4R}\setminus B_R} u^{q-s+1}\phi^\lambda \Big).
\end{aligned}
\end{equation}
Let 
$$\tau:=\frac{p+q-s+1}{2}>m-s.
$$ 
We can further estimate the integral in the right-hand side of \eqref{d6} by H\"older's inequality as follows:
\begin{equation}\label{d8}
\begin{aligned}
\int\limits_{B_{4R}\setminus B_R} u^{m-s}\phi^{\lambda-m}
&\leq \Big(\int\limits_{B_{4R}\setminus B_R} u^{\frac{p+q-s+1}{2}}\phi^{\lambda}\Big)^{\frac{m-s}{\tau}}\Big(\int\limits_{B_{4R}\setminus B_R} \phi^{\lambda-\frac{m\tau}{\tau-m+s}}\Big)^{\frac{\tau-m+s}{\tau}}\\
&\leq CR^{\frac{N(\tau-m+s)}{\tau}} \Big(\int\limits_{B_{4R}\setminus B_R} u^{\frac{p+q-s+1}{2}}\phi^{\lambda}\Big)^{\frac{m-s}{\tau}}.
\end{aligned}
\end{equation}
Combining \eqref{d6}, \eqref{d7} and \eqref{d8} we obtain
\begin{equation}\label{d9}
\begin{aligned}
\Big(\int\limits_{B_{4R}\setminus B_R} u^{\frac{p+q-s+1}{2}}\phi^\lambda \Big)^2
\leq 
CR^{\alpha-m+\frac{N(\tau-m+s)}{\tau}} \Big(\int\limits_{B_{4R}\setminus B_R} u^{\frac{p+q-s+1}{2}}\phi^{\lambda}\Big)^{\frac{m-s}{\tau}}+CR^{N+\alpha-s},
\end{aligned}
\end{equation}
for all $R>2$. Let now
$$
h(t)=t^2-CR^{\alpha-m+\frac{N(\tau-m+s)}{\tau}} t^{\frac{m-s}{\tau}}-CR^{N+\alpha-s}.
$$
We next apply Lemma \ref{lh} for 
$$
\kappa=2,\quad a=\frac{m-s}{\tau},\quad D=CR^{\alpha-m+\frac{N(\tau-m+s)}{\tau}}, \quad E=CR^{N+\alpha-s},
$$
and
$$
t=\int\limits_{B_{4R}\setminus B_R} u^{\frac{p+q-s+1}{2}}\phi^\lambda>0.
$$
To this aim, we need to check that condition \eqref{condh} is fulfilled for some $c>0$ independent of $R>2$. First, we note that
$$
\frac{E^{1-\frac{a}{\kappa}}}{D}=C^{-\frac{m-s}{2\tau}} R^{(N+\alpha-s)(1-\frac{m-s}{2\tau})-\alpha+m-N(1-\frac{m-s}{\tau})}.
$$
In order to fulfil \eqref{condh}, it is enough to check that the exponent of $R$ in the above equality is positive. Indeed, that is the case since
$$
\begin{aligned}
(N+\alpha-s)(1-\frac{m-s}{2\tau})-\alpha+m-N(1-\frac{m-s}{\tau})
&=\frac{m-s}{2\tau}(N+2\tau+s-\alpha)\\
&=\frac{m-s}{p+q-s+1}(N-\alpha+p+q+1)\\
&>0.
\end{aligned}
$$
By Lemma \ref{lh} it follows that $t<t_0<cE^{1/2}$ which yields 
$$
\int\limits_{B_{4R}\setminus B_R} u^{\frac{p+q-s+1}{2}}\phi^\lambda \leq C R^{\frac{N+\alpha-s}{2}} \quad\mbox{ for all }R>2.
$$
Since $\phi=1$ on $B_{3R}\setminus B_{2R}$, we easily deduce \eqref{d20}.
\end{proof}

\begin{lemma}\label{ls}
Let $1<s<N$ and $u$ be a positive solution of 
$$
\Lg u-\Delta_s u\geq 0\quad\mbox{ in }\Ro.
$$ 
Then, there exists $c=c(N, u, m,s)>0$ such that $u(x)\geq c|x|^{-\frac{N-s}{s-1}}$ in $\R^N\setminus B_{2}$.
\end{lemma}
\begin{proof} Let $\tilde t>0$ be given by \eqref{gg}. 
Let also $\varepsilon>0$ and $v_\varepsilon(x)=c|x|^{-\frac{N-2}{s-1}}-\varepsilon$, where $c>0$ is small such that
$$
t=|v_\varepsilon'(r)|<\tilde t\quad\mbox{ for all }|x|=r>2\quad \mbox{ and  }\quad u(x)\geq c|x|^{-\frac{N-s}{s-1}} \quad\mbox{ on }\partial B_{2}.
$$
We note that since $v_\varepsilon'(r)$ is independent of $\varepsilon$, so are $t$ and $c$ in the above estimates.
Furthermore, since $v_\varepsilon<0$ for $|x|>2$ large, we may find $R=R(\varepsilon)>2$ such that $u\geq v_\varepsilon$ in $\R^N\setminus B_R$. 

On the other hand, from \eqref{sl} and \eqref{gl}  we compute
$$
\Lg v_\varepsilon-\Delta_s v_\varepsilon=-C|x|^{-\frac{N-1}{s-1}(m-1)-1}g(t)\Big\{\frac{N-1}{s-1}\Big(m-1+\frac{tg'(t)}{g(t)}\Big)-(N-1)\Big\},
$$
where $C=\Big(\frac{N-s}{s-1}\Big)^{s-1}>0$ and $t=|v_\varepsilon'(r)|$. Thus, from \eqref{gg} we deduce 
$$
\Lg v_\varepsilon-\Delta_s v_\varepsilon\leq 0\quad\mbox{ in }\R^N\setminus B_2.
$$ 
It now follows that $u$ and $v_\varepsilon$ satisfy the hypotheses of Lemma \ref{cp} in $\Omega:=B_R\setminus B_2$. This means that $u\geq v_\varepsilon$ in $B_R\setminus B_2$ and thus,
$$
u(x)\geq  c|x|^{-\frac{N-s}{s-1}}-\varepsilon\quad\mbox{ in }\R^N\setminus B_2.
$$ 
Letting $\varepsilon\to 0$ we reach the conclusion.
\end{proof}

\begin{lemma}\label{ls1}
Let $1<s<N$, $\rho>1$ and $u$ be a positive solution of 
$$
\Lg u-\Delta_s u\geq C|x|^{-N} \quad\mbox{ in }\R^N\setminus \overline B_\rho,
$$ 
for some $C>0$.
Then, for any $0<\theta<\min\{1,\frac{1}{s-1}\}$, there exists $c=c(N, \theta, u,  m,s)>0$ such that 
$$
u(x)\geq c|x|^{-\frac{N-s}{s-1}}\log^\theta |x|\quad\mbox{ in }\R^N\setminus B_{2\rho}.
$$
\end{lemma}
\begin{proof} We fix $0<\theta<\frac{1}{s-1}$ and let $v(x)=|x|^{-\frac{N-s}{s-1}}\log^\theta |x|$. Then, with $r=|x|$ we have 
$$
\frac{rv''(r)}{v'(r)}(s-1)+N-1=\frac{\theta}{\log r}\cdot \frac{(N-s)\log r+(1-\theta)(s-1)}{\frac{N-s}{s-1}\log r-\theta},
$$
and then, from \eqref{sv} we compute
$$
\begin{aligned}
-\Delta_s v&=\frac{|v'(r)|^{s-1}}{r}\Big\{\frac{rv''(r)}{v'(r)}(s-1)+N-1\Big\}\\
&\leq C r^{-N}\log^{\theta(s-1)-1} r \cdot \frac{(N-s)\log r+(1-\theta)(s-1)}{\frac{N-s}{s-1}\log r-\theta}\\
&\leq C r^{-N}\quad\mbox{ for $r>\rho$ large}.
\end{aligned}
$$
For large $r=|x|>\rho$ we have 
$$
m-1+\frac{tg'(t)}{g(t)}>s-1 \quad\mbox{ where }t=|v'(r)|\, ,
$$
and thus from \eqref{gv} and the fact that $g$ is bounded we obtain
$$
\begin{aligned}
\Lg v&=\frac{t^{m-1}g(t)}{r}\Big\{\frac{rv''(r)}{v'(r)}\Big(m-1+\frac{tg'(t)}{g(t)}\Big)+N-1\Big\}\\
&\leq c\frac{t^{m-1}}{r}\Big\{\frac{rv''(r)}{v'(r)}(s-1)+N-1\Big\}\\
&\leq C r^{-N}.
\end{aligned}
$$
Thus, for $c>0$ small we deduce
\begin{equation}\label{rrr}
\Lg (cv)-\Delta_s (cv)\leq C|x|^{-N} \quad\mbox{ in }\R^N\setminus B_{2\rho},
\end{equation}
 and $cv\leq u$ on $\partial B_{2\rho}$. 
Next, we follow a similar approach to that in the proof of Lemma \ref{ls}. Precisely, for $0<\varepsilon<1$ we let $v_{\varepsilon}=cv-\varepsilon$. Then, since $u>0$, we can find $R>2\rho$ such that $u\geq v_\varepsilon$ in $\R^N\setminus B_{R}$. From \eqref{rrr} and $u\geq v_\varepsilon$ on $\partial (B_{R}\setminus B_{2\rho})$ we deduce that $u$ and $v_\varepsilon$ satisfy the hypotheses in Lemma \ref{cp} for $\Omega=B_{R}\setminus B_{2\rho}$. It then follows that $u\geq v_\varepsilon$ in $B_{R}\setminus B_{2\rho}$ which further implies that $u\geq v_\varepsilon$ in $\R^N\setminus B_{2\rho}$. Letting $\varepsilon\to 0$ we deduce $u\geq cv$ in $\R^N\setminus B_{2\rho}$.  This concludes our proof.
\end{proof}

We are now ready to proceed with the proof of Theorem \ref{tm2}. Since (ii)$\Longrightarrow$(i) is obvious, it is enough to complete the following sequence of implications: (i)$\Longrightarrow$(iii)$\Longrightarrow$(ii).

\noindent{\bf Proof of (i)$\Longrightarrow$(iii).} Assume that $u$ is a positive solution of  $(P^-)$. We divide our argument into three steps.

\smallskip

\noindent{\bf Step 1: } $p> \frac{(s-1)(N-\alpha)}{N-s}$.

By Lemma \ref{ls} we have $u(x)\geq c|x|^{-\frac{N-s}{s-1}}$ in $\R^N\setminus B_{2}$.

If $p\leq \frac{(s-1)(N-\alpha)}{N-s}$, then
$$
\begin{aligned}
\int\limits_{|y|>2}\frac{u^{p}(y)}{1+|y|^{\alpha}}dy&\geq \int\limits_{|y|>2}\frac{u^{p}(y)}{2|y|^{\alpha}}dy \geq 
C\int\limits_{|y|>2}|y|^{-\frac{p(N-s)}{s-1}-\alpha}dy\\
&=C\int_{2}^\infty t^{N-\frac{p(N-s)}{s-1}-\alpha}=\infty,
\end{aligned}
$$
which contradicts condition \eqref{c1} and implies $p> \frac{(s-1)(N-\alpha)}{N-s}$.

\smallskip

\noindent{\bf Step 2: } $p+q>\frac{(s-1)(2N-\alpha)}{N-s}$.

We combine Lemma \ref{lap} and Lemma \ref{ls}. Thus, by \eqref{d20} and 
$u(x)\geq c|x|^{-\frac{N-s}{s-1}}$ in $\R^N\setminus B_{2}$ we estimate 
$$
\begin{aligned}
CR^{\frac{N+\alpha-s}{2}} &\geq  \int\limits_{B_{3R}\setminus B_{2R}}u^{\frac{p+q-s+1}{2}}(x)dx \\
&\geq C \int\limits_{B_{3R}\setminus B_{2R}}|x|^{-\frac{(N-s)(p+q-s+1)}{2(s-1)}}dx\\
&=CR^{N- \frac{(N-s)(p+q-s+1)}{2(s-1)}}.
\end{aligned}
$$
This implies
$$
\frac{N+\alpha-s}{2}\geq N- \frac{(N-s)(p+q-s+1)}{2(s-1)},
$$
which yields $p+q\geq\frac{(s-1)(2N-\alpha)}{N-s}$. It remains to rule out the equality case. Assume that $p+q=\frac{(s-1)(2N-\alpha)}{N-s}$ and $p> \frac{(s-1)(N-\alpha)}{N-s}$. Then, by Lemma \ref{lbas} (i) with $f(x)=u^p(x)\geq c|x|^{-\beta}$ and $\beta=\frac{p(N-s)}{s-1}>N-\alpha$, we have 
$$
|x|^{-\alpha}*u^p\geq c|x|^{N-\alpha-\frac{p(N-s)}{s-1}}\quad\mbox{ in }\R^N\setminus B_2.
$$
Thus, $u$ satisfies
$$
\Lg u-\Delta_s u\geq c|x|^{N-\alpha-\frac{(p+q)(N-s)}{s-1}}=c|x|^{-N}\quad\mbox{ in }\R^N\setminus B_2.
$$
By Lemma \ref{ls1} (with $\rho=2$) it follows that
\begin{equation}\label{dddd}
u(x)\geq c|x|^{-\frac{N-s}{s-1}}\log^\theta |x|\quad\mbox{ in }\R^N\setminus B_{4},
\end{equation}
for some $0<\theta<1$ small. With this improved inequality we return to \eqref{d20} to reach a contradiction. Indeed, we have
$$
\begin{aligned}
CR^{\frac{N+\alpha-s}{2}} &\geq  \int\limits_{B_{3R}\setminus B_{2R}}u^{\frac{p+q-s+1}{2}}(x)dx \\
&\geq C \int\limits_{B_{3R}\setminus B_{2R}}|x|^{-\frac{(N-s)(p+q-s+1)}{2(s-1)}}\big(\log |x|\big)^{\frac{\theta(p+q-s+1)}{2}} dx\\
&\geq C \big(\log 2R\big)^{\frac{\theta(p+q-s+1)}{2}}\int\limits_{B_{3R}\setminus B_{2R}}|x|^{-\frac{N-\alpha+s}{2}} dx\\
&=CR^{\frac{N+\alpha-s}{2}}(\log 2R\big)^{\frac{\theta(p+q-s+1)}{2}},
\end{aligned}
$$
which is impossible to hold for $R>4$ large. Hence, $p+q>\frac{(s-1)(2N-\alpha)}{N-s}$.

\smallskip

\noindent{\bf Step 3: } $q> \frac{(s-1)(N-\alpha)}{N-s}$.

We first note that for $|x|>1$ and $2\leq |y|\leq 3$ we have $|x-y|\leq |x|+3< 4|x|$. Thus, for $|x|>1$ we deduce
$$
|x|^{-\alpha}*u^p\geq \int\limits_{2\leq |y|\leq 3}\frac{u^p(y)}{|x-y|^\alpha} dy\geq C|x|^{-\alpha}\int\limits_{2\leq |y|\leq 3}u^p(y)dy=C|x|^{-\alpha},
$$
and thus $u$ satisfies
\begin{equation}\label{uq}
\Lg u-\Delta_s u\geq c|x|^{-\alpha}u^q \quad\mbox{ in }\Ro.
\end{equation}
In the case $\alpha=0$ and $\Lg u=-\Delta_mu$, the inequality \eqref{uq} was discussed in the recent paper \cite{BBF26}. In turn, we shall follow the current approach based on the a priori estimates we developed in Lemma \ref{lap}. Let $\phi$ be the standard cut-off function we considered in the proof of Lemma \ref{lap}. We use $u^{1-s}\phi^\lambda$ as a test function in the inequality $(P^-)$ which, instead of \eqref{d5}, yields 
\begin{equation}\label{e1}
\int\limits_{B_{4R}\setminus B_R} u^{q-s+1}\phi^\lambda \leq CR^{\alpha-m}\int\limits_{B_{4R}\setminus B_R} u^{m-s}\phi^{\lambda-m} +CR^{N-s+\alpha},
\end{equation}
for all $R>2$. By H\"older's inequality we have
\begin{equation}\label{e2}
\begin{aligned}
\int\limits_{B_{4R}\setminus B_R} u^{m-s}\phi^{\lambda-m}
&\leq \Big(\int\limits_{B_{4R}\setminus B_R} u^{q-s+1}\phi^{\lambda}\Big)^{\frac{m-s}{q-s+1}}\Big(\int\limits_{B_{4R}\setminus B_R} \phi^{\lambda-\frac{m(q-s+1)}{q-m+1}}\Big)^{\frac{q-m+1}{q-s+1}}\\
&\leq CR^{\frac{N(q-m+1)}{q-s+1}} \Big(\int\limits_{B_{4R}\setminus B_R} u^{q-s+1}\phi^{\lambda}\Big)^{\frac{m-s}{q-s+1}}.
\end{aligned}
\end{equation}
From \eqref{e1} and \eqref{e2} we deduce
\begin{equation}\label{e3}
\int\limits_{B_{4R}\setminus B_R} u^{q-s+1}\phi^\lambda \leq CR^{\alpha-m+\frac{N(q-m+1)}{q-s+1}} \Big(\int\limits_{B_{4R}\setminus B_R} u^{q-s+1}\phi^{\lambda}\Big)^{\frac{m-s}{q-s+1}}+CR^{N-s+\alpha},
\end{equation}
for all $R>2$. Set 
$$
h(t)=t-CR^{\alpha-m+\frac{N(q-m+1)}{q-s+1}} t^{\frac{m-s}{q-s+1}}-CR^{N+\alpha-s},
$$
and apply Lemma \ref{lh} for 
$$
\kappa=1,\quad a=\frac{m-s}{q-s+1},\quad D=CR^{\alpha-m+\frac{N(q-m+1)}{q-s+1}}, \quad E=CR^{N+\alpha-s},
$$
and
$$
t=\int\limits_{B_{4R}\setminus B_R} u^{q-s+1}\phi^\lambda>0.
$$
To this aim, we need to check that condition \eqref{condh} is fulfilled for some $c>0$ independent of $R>2$. We note that 
$$
\begin{aligned}
\frac{E^{1-\frac{a}{\kappa}}}{D}&=C^{\frac{q-m+1}{q-s+1}} R^{\frac{(N+\alpha-s)(q-m+1)}{q-s+1}-\alpha+m-\frac{N(q-m+1)}{q-s+1}}\\
&=C^{\frac{q-m+1}{q-s+1}} R^{\frac{(\alpha-s)(q-m+1)}{q-s+1}-\alpha+m}.
\end{aligned}
$$
Then \eqref{condh} clearly holds if the exponent of $R$ in the above equalities is nonnegative, that is,
$$
\frac{(\alpha-s)(q-m+1)}{q-s+1}-\alpha+m\geq 0
$$
The above inequality is equivalent to $q\geq \alpha-1$ which is fulfilled according to our hypothesis on $q$. By Lemma \ref{lh} it follows that 
$$
\int\limits_{B_{4R}\setminus B_R} u^{q-s+1}\phi^\lambda \leq C E,
$$
which yields
\begin{equation}\label{e4}
\int\limits_{B_{3R}\setminus B_{2R}} u^{q-s+1} \leq CR^{N-s+\alpha} \quad\mbox{ for all $R>2$ large.}
\end{equation}
Since from Lemma \ref{ls} we have $u(x)\geq c|x|^{-\frac{N-s}{s-1}}$ in $\R^N\setminus B_{2}$, from \eqref{e4} we infer that
$$
CR^{N-s+\alpha}\geq \int\limits_{B_{3R}\setminus B_{2R}} |x|^{-\frac{(q-s+1)(N-s)}{s-1}}=CR^{N-\frac{(q-s+1)(N-s)}{s-1}}, 
$$
for all $R>2$ large. This implies
$$
N-s+\alpha\geq N-\frac{(q-s+1)(N-s)}{s-1}\Longrightarrow q\geq \frac{(s-1)(N-\alpha)}{N-2}.
$$
To conclude the proof of this part, it remains to show that $q\neq \frac{(s-1)(N-\alpha)}{N-s}$.

Assume by contradiction that $q=\frac{(s-1)(N-\alpha)}{N-s}$. Then, by the estimate in Lemma \ref{ls} and \eqref{uq} we have 
$$
\Lg u-\Delta_s u\geq |x|^{-\alpha}u^q\geq c|x|^{-\alpha-\frac{q(N-s)}{s-1}}=c|x|^{-N}\quad\mbox{ in }\Ro.
$$
From now on, we proceed as in Step 2 above. Precisely, from Lemma \ref{ls1} we have that $u$ fulfills \eqref{dddd} for some $0<\theta<1$ small, which used in \eqref{e4} leads to a contradiction. This completes the proof of the implication (i)$\Longrightarrow$(iii).

\smallskip

\noindent{\bf Proof of (iii)$\Longrightarrow$(ii).} Assume now that conditions \eqref{pq} hold and let us construct a positive solution $u$ of $(P^-)$.
To this aim, we devise a sub and supersolution method suitable to our setting which features nonlocal terms given by the convolution in an unbounded domain.

To construct a supersolution we first rewrite \eqref{pq} as
$$
\begin{cases}
\displaystyle p>\frac{(s-1)(N-\alpha)}{N-s}\, , \\[0.1in]
\displaystyle q-s+1>\frac{(s-1)(s-\alpha)}{N-s}\, , \\[0.1in]
\displaystyle p+q-s+1>\frac{(s-1)(N-\alpha+s)}{N-s}.
\end{cases}
$$
Thus, we may find $0<\gamma<\frac{N-s}{s-1}$ such that $p\gamma\neq N$ and
\begin{equation}\label{pqc}
p\gamma >N-\alpha, \quad 
(q-s+1)\gamma >s-\alpha, \quad
(p+q-s+1)\gamma>N-\alpha+s.
\end{equation}
The condition $p\gamma\neq N$ allows us to avoid the presence of logarithmic terms when using Lemma \ref{lbas} (ii). Letting now $0<\kappa<1$ and
\begin{equation}\label{sup1}
\overline u(x):=\kappa |x|^{-\gamma},
\end{equation} 
we see from Lemma \ref{lbas} (ii) with $\beta=p\gamma\neq N$  that
\begin{equation}\label{lbas1}
(|x|^{-\alpha}*\overline u^p)\overline u^q(x)\leq C_1 \kappa^{p+q}
\left\{
\begin{aligned}
& |x|^{N-\alpha-(p+q)\gamma} &&\quad\mbox{ if }\; N-\alpha<p\gamma<N\\[0.1in]
& |x|^{-\alpha-q\gamma} &&\quad\mbox{ if }\; p\gamma>N
\end{aligned}
\right.
\quad\mbox{in }\R^N\setminus \overline B_{1},
\end{equation}
where $C_1>0$ is a constant independent of $\kappa>0$.
Furthermore, from \eqref{sv} and \eqref{gv} we have 
\begin{equation}\label{sub11}
\Lg \overline u-\Delta_s \overline u\geq C_2\kappa^{s-1} |x|^{-(\gamma+1)(s-1)-1} 
\quad\mbox{in }\R^N\setminus \overline B_{1},
\end{equation}
where $C_2>0$ is a constant independent of $\kappa>0$. Hence, for $\kappa>0$ small enough we see from \eqref{lbas1} and \eqref{sub11} that $\overline u$ satisfies
$$
\Lg \overline u-\Delta_s \overline u\geq (|x|^{-\alpha}*\overline u^p)\overline u^q
\quad\mbox{in }\R^N\setminus \overline B_{1}.
$$

As a subsolution to \eqref{eqpq} we take
\begin{equation}\label{sub1}
\underline u(x):= c|x|^{-\frac{N-s}{s-1}}\, ,
\end{equation}
where $0<c<\kappa$. 
We have seen in the proof of Lemma \ref{ls} that for $c>0$ small, $\underline u$ satisfies
\begin{equation}
\Lg \underline u-\Delta_s \underline u\leq 0\quad\mbox{ in }\R^N\setminus\overline B_1.
\end{equation}
Since $0<c<\kappa<1$ and $0<\gamma<\frac{N-s}{s-1}$ we deduce $\overline u>\underline u$ in $\Ro$.

Let next $u_1=\underline u$ and for any $n\geq 2$ we consider the problem
\begin{equation}\label{iter1}
\begin{cases}
\Lg u_n-\Delta_s u_n  \displaystyle =\Big(\int\limits_{B_n\setminus B_1}\frac{u_{n-1}^p(y)}{|x-y|^\alpha}\Big)u_{n-1}^q \quad\mbox{ in }B_n\setminus \overline B_1,\\[0.1in]
u_n  =\underline u \quad\mbox{ in }\R^N\setminus B_n.
\end{cases}
\end{equation}

\smallskip

\noindent{\bf Step 1:} Existence and uniqueness of $u_n\in C^1(B_n\setminus \overline B_1)$, $n\geq 2$. Furthermore, we have $\underline u\leq u_{n-1}\leq u_{n}\leq \overline u$ in $B_{n}\setminus B_1$.

In proving the above claim, we first need the result below. 

\begin{lemma}\label{LG}
Let $\Omega\subset \Ro$ be a smooth bounded domain and $f\in L^\infty(\Omega)$ be such that 
\begin{equation}\label{ff0}
0\leq f(x)\leq (|x|^{-\alpha}*\overline u^p)\overline u^q(x)\quad\mbox{  in }\Omega.
\end{equation}
Then, the problem
\begin{equation}\label{iter2}
\begin{cases}
\Lg u-\Delta_s u  \displaystyle =f(x) \quad\mbox{ in }\Omega\, ,\\
u  =\underline u \quad\mbox{ on }\partial\Omega
\end{cases}
\end{equation}
\end{lemma}
has a unique solution $u\in C^1(\Omega)\cap W^{1,m}(\Omega)$ and $\underline u\leq u\leq \overline u$ in $\Omega$.

\begin{proof} Let $G:\R\to \R$ be such that $G(0)=0$ and $G'(t)=t|t|^{m-2}g(t)+t|t|^{s-2}$. By \eqref{cong1} we have that $G$ is convex.  We also claim that there exist $c_2>c_1>0$ such that 
\begin{equation}\label{GG}
c_1(t^m g(t)+t^s)\leq G(t) \leq c_2(t^m g(t)+t^s)\quad\mbox{ for all }t\geq 0.
\end{equation}
Indeed, integrating by parts, for all $t>0$ we have 
$$
G(t)=\frac{t^s}{s}+\int_0^t \tau^{m-1}g(\tau) d\tau=\frac{t^s}{s}+\frac{t^mg(t)}{m}-\frac{1}{m}\int_0^t \tau^{m}g'(\tau) d\tau.
$$
Since $g'\leq 0$ the above equality implies the first estimate in \eqref{GG}. 
Furthermore, using \eqref{cong1} we have $\tau^{m}g'(\tau)\geq -(m-1)\tau^{m-1}g(\tau)$ so
$$
G(t)\leq \frac{t^s}{s}+\frac{t^mg(t)}{m}+\frac{m-1}{m}\int_0^t \tau^{m-1}g(\tau) d\tau\leq \frac{t^s}{s}+\frac{t^mg(t)}{m}+\frac{m-1}{m}G(t),
$$
which gives
$$
\frac{1}{m}G(t)\leq \frac{t^s}{s}+\frac{t^mg(t)}{m} \quad\mbox{ for all }t\geq 0
$$
and this implies the second estimate in \eqref{GG}.

Denote by $W^{1, G}(\Omega)$ the set of all measurable functions $u:\Omega\to \R$ for which $D_{i} u:=\frac{\partial u}{\partial x_i}$ exists in the weak sense, $1\leq i\leq N$, and 
$$
\int_\Omega G(|\nabla u|) dx<\infty.
$$
Then, $W^{1, G}(\Omega)$ becomes a reflexive Banach space (see \cite[Theorem 8.28]{A75}) with respect to the norm (see \cite[Lemma 5.7 ]{G74})
$$
\|u\|_G=\inf\left\{k>0: \int_\Omega G\Big(\frac{|D_{i} u|}{k}\Big)\leq 1, \mbox{ for all } 1\leq i\leq N \right\}.
$$
Let also $W_0^{1, G}(\Omega)$ be the closure of $C_0^\infty(\Omega)$ with respect to the $\|\cdot \|_G$ norm. Set
$$
I:\underline u+W_0^{1, G}(\Omega)\to \R, \quad I(u)=\int_\Omega G(|\nabla u|)dx-\int_\Omega f(x) u.
$$
Then $I$ is of class $C^1$ and 
$$
\langle I'(u), \phi \rangle=\int\limits_\Omega G'(|\nabla u|)\frac{\nabla u}{|\nabla u|}\cdot\nabla \phi-\int\limits_\Omega f(x) \phi,\quad\mbox{ for all }\phi\in C_0^\infty(\Omega).
$$
Since $G(t)\geq \frac{1}{s}t^s$, we easily derive that $I$ is coercive and thus has a critical point which is a solution of \eqref{iter2}. From Lemma \ref{cp} we have $\underline u\leq u\leq \overline u$ in $\Omega$. By \cite[Theorem 1.7]{L91} we derive $u\in C^1(\Omega)$. 
\end{proof}
To conclude the proof in this step, it remains to apply  Lemma \ref{LG} for $\Omega=B_n\setminus \overline B_1$ and 
$$
0\leq f(x):=\Big(\int\limits_{B_n\setminus B_1}\frac{u_{n-1}^p(y)}{|x-y|^\alpha}\Big)u_{n-1}^q(x) \quad\mbox{ in }B_n\setminus \overline B_1.
$$
Starting from $u_1=\underline u$ we see that condition \eqref{ff0} is fulfilled since $\underline u\leq \overline u$ in $B_2\setminus B_1$. Then, since  $u_1=\underline u\leq u_2\leq \overline u$ in $B_2\setminus B_1$, we may  apply Lemma \ref{LG} to derive 
$\underline u\leq u_2\leq u_3\leq \overline u$ in $B_3\setminus B_1$ and we may further proceed by induction. 
\smallskip

\noindent{\bf Step 2:}  The sequence $\{u_n\}_{n\geq 2}$ is bounded in $W^{1, G}_{loc}(\Ro)$.

Let $K\subset \Ro$ be a compact set and take $R\geq 1$ an integer such that $K\subset B_R\setminus \overline B_1$. Take $\phi\in C_0^\infty(B_R\setminus B_1)$ such that $0\leq \phi\leq 1$ and $\phi=1$ on $K$. We test the problem \eqref{iter1} with $\phi^k u_n$, $k>m>s$. Using the fact that $u_{n-1}\leq \overline u$ in $B_R\setminus B_1$, we find  
$$
\int\limits_{B_R\setminus B_1}\Big(|\nabla u_n|^{m-2}g(|\nabla u_n|)+|\nabla u_n|^{s-2}\Big) \nabla u_n\cdot \nabla (\phi^k u_n) \leq  \int\limits_{B_R\setminus B_1}(|x|^{-\alpha}*\underline u^p)\underline u^q (\phi^k \underline u)=C<\infty.
$$
This yields
\begin{equation}\label{ee1}
\begin{aligned}
\int\limits_{B_R\setminus B_1}\Big(|\nabla u_n|^{m} & g(|\nabla u_n|) +|\nabla u_n|^{s}\Big)  \phi^k \\
&\leq  C+\int\limits_{B_R\setminus B_1}\Big(|\nabla u_n|^{m-1}g(|\nabla u_n|)+|\nabla u_n|^{s-1}\Big)  |\nabla \phi^k| \overline u\\
&\leq  C+C_1\int\limits_{B_R\setminus B_1}\Big(|\nabla u_n|^{m-1}g(|\nabla u_n|)+|\nabla u_n|^{s-1}\Big)  |\nabla \phi^k|.
\end{aligned}
\end{equation}
We note that $|\nabla \phi^k|\leq c\phi^{k-1}$, so by Young's inequality we derive
\begin{equation}\label{ee2}
\begin{aligned}
|\nabla u_n|^{s-1}  |\nabla \phi^k| & \leq \frac{1}{2C_1} |\nabla u_n|^{s} \phi^{\frac{s(k-1)}{s-1}}+C(s,k)\\
& \leq \frac{1}{2C_1} |\nabla u_n|^{s} \phi^{k}+C(s,k).
\end{aligned}
\end{equation}
Similarly we have 
$$
|\nabla u_n|^{m-1}  |\nabla \phi^k| \leq \frac{1}{2C_1} |\nabla u_n|^{m} \phi^{k}+C(m,k).
$$
Since $g(|\nabla u_n|)\leq g(0)$ from the last estimate we obtain
\begin{equation}\label{ee3}
|\nabla u_n|^{m-1} g(|\nabla u_n|) |\nabla \phi^k| \leq \frac{1}{2C_1} |\nabla u_n|^{m} g(|\nabla u_n|)\phi^{k}+C(m,k)g(0).
\end{equation}
We now use \eqref{ee2} and \eqref{ee3} into \eqref{ee1} and the fact that $\phi=1$ on $K$ to deduce 
\begin{equation}\label{ee4}
\begin{aligned}
\int\limits_{K}\Big(|\nabla u_n|^{m}  g(|\nabla u_n|) +|\nabla u_n|^{s}\Big)  &\leq \int\limits_{B_R\setminus B_1}\Big(|\nabla u_n|^{m} & g(|\nabla u_n|) +|\nabla u_n|^{s}\Big)  \phi^k \\
&\leq  2C.
\end{aligned}
\end{equation}
Using the estimates \eqref{GG}, it follows that $\{u_n\}_{n\geq R}$ is bounded in $W^{1, G}(K)$.

\smallskip

\noindent{\bf Step 3:} The function $u(x):=\lim_{n\to \infty}u_n(x)$, $x\in \Ro$ is a $C^1$-solution of \eqref{eqpq}.

Let us note that $u$ is well defined as it is the pointwise limit of the monotone sequence $\{u_n\}_{n\geq 1}$. Let $\{\varepsilon_n\}_{n\geq 2}$ be a decreasing sequence that converges to zero and $0<\varepsilon_n<1$. Since $\{u_n\}_{n\geq 2}$ is bounded in $W^{1, G}(B_2\setminus B_{1+\varepsilon_2})$, we may find a subsequence $\{u_n^2\}_{n\geq 2}$ which converges to $u$ in $W^{1, G}(B_2\setminus B_{1+\varepsilon_2})$. Next, since the sequence $\{u_n^2\}_{n\geq 3}$ is bounded in $W^{1, G}(B_3\setminus B_{1+\varepsilon_3})$, we may find a subsequence $\{u_n^3\}_{n\geq 3}$ which converges to $u$ in $W^{1, G}(B_3\setminus B_{1+\varepsilon_3})$ and so on. Inductively, for all $k\geq 2$ we construct a sequence  
$\{u_n^{k-1}\}_{n\geq k}$ which is bounded in $W^{1, G}(B_k\setminus B_{1+\varepsilon_k})$, and thus admits a subsequence $\{u_n^k\}_{n\geq k}$ which converges to $u$ in $W^{1, G}(B_k\setminus B_{1+\varepsilon_k})$. It follows that $\{u_n^n\}_{n\geq 2}$ converges to $u$ in $W^{1, G}_{loc}(\Ro)$. We claim that $u$ is a weak solution of \eqref{eqpq}. Indeed, let $\phi\in C_0^\infty(\Ro)$ and take a positive integer $R\geq 2$ such that supp$\,\phi\subset B_R\setminus B_{1+\varepsilon_R}$. From \eqref{iter1}, for all $n\geq R$ we have
\begin{equation}\label{nn}
\int\limits _{B_R\setminus B_{1+\varepsilon_R}}\Big(|\nabla u_n^n|^{m-2}g(|\nabla u_n^n|)+|\nabla u_n^n|^{s-2}\Big) \nabla u_n^n\cdot \nabla \phi = \int\limits _{B_R\setminus B_{1+\varepsilon_R}} \big(|x|^{-\alpha}*(u_n^n)^p\big)(u_n^n)^q \phi.
\end{equation}
The sequence $\{u_n^n\}_{n\geq R}$ converges to $u$ in $W^{1, G}(B_R\setminus B_{1+\varepsilon_R})$. Also, since $u_n^n\leq \overline u$ in $B_R\setminus B_1$, we may apply the Lebesgue theorem to deduce the convergence of the right-hand side of \eqref{nn}. Thus, letting $n\to \infty$ in \eqref{nn} we deduce that $u$ is a $W^{1, m}_{loc}(\Ro)$ solution of \eqref{eqpq}. Finally, since $\underline u\leq u_n^n\leq \overline u$ in $B_n\setminus \overline B_1$, we find $\underline u\leq u\leq \overline u$ in $\Ro$. In particular $u\in L^\infty_{loc}(\Ro)$ and now Theorem 1.7 in \cite{L91} yields $u\in C^1(\Ro)$. This concludes our proof.
\qed

\section*{Acknowledgement} The author would like to thank the anonymous Referee for the careful reading of the manuscript and for pointing out several improvements that led to the present version.


\begin{thebibliography}{99}

\bibitem{A75} R. Adams, Sobolev Spaces, Academic Press, New York, 1975.

\bibitem{BBF26} M. Bhakta, A. Biswas and R. Filippucci, 
Liouville properties for differential inequalities with $(p,q)$ Laplacian operator, {\it  J. Lond. Math. Soc.} {\bf 113} (2026) e70490. 

\bibitem{CJJ26} L. Cai, P. Ji and P. Jin, Concentration of solutions with prescribed mass for 
$(p,q)$-Laplacian equations with nonlocal reaction,  {\it Bull. Math. Sci.} {\bf 16} (2026) 2550024.

\bibitem{CPR26} L. Cai, N.S. Papageorgiou and V.D. R\u adulescu, Multiple solutions with sign information for non-autonomous, non-coercive $(p,q)$-equations,  {\it Advances in Differential Equations} {\bf 31} (2026) 723-748.

\bibitem{F18} G.M. Figueiredo, A. Moussaouib, G.C.G. dos Santos and L. Tavares, A sub-supersolution approach for some classes of nonlocal problems involving Orlicz spaces, {\it J. Differential Equations} {\bf 267} (2019) 4148-4169.

\bibitem{FG20} R. Filippucci and M. Ghergu,  Singular solutions for coercive quasilinear elliptic inequalities with nonlocal terms, {\it Nonlinear Anal.} {\bf 197} (2020) 111857.


\bibitem{FG22a} R. Filippucci and M. Ghergu, Higher order evolution inequalities with nonlinear convolution terms, 
{\it Nonlinear Anal.} {\bf 221} (2022) 112881.

\bibitem{FG22b} R. Filippucci and M. Ghergu, Fujita type results for quasilinear parabolic inequalities with nonlocal terms, {\it Discrete Contin. Dyn. Syst.} {\bf 42} (2022)  1817-1833.

\bibitem{GKS20} M. Ghergu, P. Karageorgis and G. Singh, Positive solutions for quasilinear elliptic inequalities and systems with nonlocal terms, {\it J. Differential Equations} {\bf 268} (2020) 6033-6066. 

\bibitem{Gbook22} M. Ghergu, Partial Differential Inequalities with Nonlinear Convolution Terms, Springer Briefs in Mathematics, 2022.

\bibitem{GT16book} M. Ghergu and S.D. Taliaferro,  Isolated Singularities in Partial Differential Inequalities, Cambridge University Press, 2016.

\bibitem{GZ23} M. Ghergu and Z. Yu, Elliptic inequalities with nonlinear convolution and Hardy terms in cone-like domains, {\it J. Math. Anal. Appl.} {\bf 526} (2023) 127329.

\bibitem{GZ26} M. Ghergu and Z. Yu, Isolated singularities for elliptic equations with convolution terms in a punctured ball, {\it J. Geom. Anal.} {\bf 36} (2026) 16.

\bibitem{G74} J.P. Gossez, Nonlinear elliptic boundary value problems for equations with rapidly (or slowly) increasing coefficients, {\it Trans. Amer. Math. Soc.} {\bf 190} (1974) 163-205.


\bibitem{Har28a} D.R. Hartree, The wave mechanics of an atom with a non-Coulomb central field, Part I. Theory and Methods, {\it Math.  Proc. Cambridge Phil. Soc.} {\bf 24} (1928) 89-110. 

\bibitem{Har28b} D.R. Hartree, The wave mechanics of an atom with a non-Coulomb central field, Part II. Some Results and Discussion, {\it Math.  Proc. Cambridge Phil. Soc.} {\bf 24} (1928) 111-132. 

\bibitem{Har28c} D.R. Hartree, The wave mechanics of an atom with a non-Coulomb central field, Part III. Term Values and Intensities in Series in Optical Spectra, {\it Math.  Proc. Cambridge Phil. Soc.} {\bf 24} (1928) 426-437.

\bibitem{L91} G. M. Lieberman, The natural generalization of the natural conditions of Ladyzhenskaya and Ural’tseva for
elliptic equations, {\it Comm. Partial Differential Equations} {\bf 16} (1991) 311-361.

\bibitem{LR26} H. Liu and V.D. R\u adulescu, Normalized solutions of quasilinear problems
with nonlocal reaction, {\it  Bull. Math. Sci.}  {\bf  16} (2026) 2650006. 

\bibitem{LZ26} S. Long and Z. Wang, Existence and asymptotic behavior of ground states for a nonlocal magnetic system, {\it Advances in Differential Equations} {\bf 31} (2026) 417-436.

\bibitem{MNV24} A. Moussaoui, D. Nabab and J. V\'elin, Singular quasilinear convective systems involving variable exponents, {\it Opuscula Math.} {\bf 44} (2024) 105-134.

\bibitem{NS26} N. Nidhi and K. Sreenadh, Solutions with prescribed mass 
for a critical Choquard equation driven by a local-nonlocal operator, 
{\it Opuscula Math.} {\bf 46} (2026) 235-266.

\bibitem{TZ26} Y. Tao and J. Zhang, 
Normalized ground states for a $p$-Laplacian system in the mass super-critical case, 
{\it Opuscula Math.} {\bf 46} (2026) in press.

\bibitem{WS25} L. Wei and  Y. Song, Normalized solutions for critical Schrödinger equations involving $(2,q)$-Laplacian, {\it Opuscula Math.} {\bf 45} (2025) 685-716. 

\bibitem{Z26} H. Zhang, Double phase problems with competing potentials
and asymptotically linear reaction, {\it  Bull. Math. Sci.} {\bf  16} (2026) 2650001. 

\bibitem{ZLP25} S. Zeng, Y. Lu and N. S. Papageorgiou, Double phase elliptic inclusions with convection and logarithmic perturbed terms, {\it Bull. Math. Sci.} {\bf 15} (2025)  2550020. 

\end{thebibliography}
\end{document}